\documentclass{article}

\usepackage{amsmath}
\usepackage{amssymb}
\usepackage{amsthm}
\usepackage{amsrefs}
\usepackage{color}
\usepackage{dsfont}
\usepackage{mathrsfs}
\usepackage{stmaryrd}
\usepackage[all]{xy}
\usepackage[mathcal]{eucal}
\usepackage{tikz-cd}
\usepackage{mathtools}
\usepackage{verbatim}  
\usepackage{fullpage}  
\usepackage{soul,xcolor}
\usepackage{float}
\usepackage{caption}

\newcommand{\N}{\mathbb{N}}

\newcommand{\R}{\mathbb{R}}
\newcommand{\Z}{\mathbb{Z}}

\renewcommand{\phi}{\varphi}

\newcommand{\excise}[1]{}

\newcommand{\abs}[1]{\left \vert {#1} \right \vert}
\newcommand{\s}[1]{\left\{ {#1} \right\}}

\newcommand{\parens}[1]{\left( {#1} \right)}
\newcommand{\defeq}{\vcentcolon=}

\newtheorem{theorem}{Theorem}

\newtheorem{proposition}[theorem]{Proposition}
\newtheorem{lemma}[theorem]{Lemma}
\newtheorem{corollary}[theorem]{Corollary}

\theoremstyle{definition}

\title{Geometry of non-transitive graphs}
\author{Josiah Oh and Mark Pengitore}

\begin{document}

\maketitle

\begin{abstract}
In this note, we study non-transitive graphs and prove a number of results when they satisfy a coarse version of transitivity. Also, for each finitely generated group $G$, we produce continuum many pairwise non-quasi-isometric regular graphs that have the same growth rate, number of ends, and asymptotic dimension as $G$.
\end{abstract}

\section{Introduction}
Woess \cite{We} asked the following natural question: does every transitive graph ``look like" a Cayley graph? To be more precise, is every  connected, locally finite, vertex-transitive graph quasi-isometric to a Cayley graph of some finitely generated group? In  \cite{DL}, Diestel and Leader constructed infinite, vertex-transitive graphs of exponential growth, denoted $\text{DL}(m,n)$, and conjectured that these graphs would provide a negative answer to Woess' question. In a series of papers, Eskin, Fisher, and Whyte \cites{EFW1, EFW2, EFWA} confirmed this by demonstrating that $\text{DL}(m,n)$ is not quasi-isometric to any finitely generated group when $m \neq n$. They also constructed a class of non-unimodular, three-dimensional, solvable, non-nilpotent Lie groups that do not admit a nonpositively curved, left-invariant metric, and showed that these groups are not quasi-isometric to any finitely generated group.
 
The spaces in both of the above collections all have exponential volume growth, which leads us to ask: does there exist a nilpotent Lie group or a vertex-transitive graph of polynomial growth that is not quasi-isometric to any finitely generated group? We note that there are uncountably many pairwise non-quasi-isometric Carnot groups (recall that Carnot groups are examples of simply connected nilpotent Lie groups). We also note that any finitely generated group which is quasi-isometric to a nilpotent Lie group is virtually nilpotent by Gromov's polynomial growth theorem \cite{G}, and hence is quasi-isometric to a finitely generated nilpotent group. Now, all finitely generated nilpotent groups are finitely presented groups and there are only countably many of these. Therefore, by a counting argument, there exist many simply connected nilpotent Lie groups that are not quasi-isometric to any finitely generated group. As for locally finite vertex-transitive graphs of polynomial growth, Trofimov \cite{VIT} had already demonstrated, even before Woess asked his question, that such graphs are quasi-isometric to finitely generated nilpotent groups.
 
Given the above discussion, we choose to look beyond the world of vertex-transitive graphs. Since the class of non-vertex-transitive graphs is so large, one expects that there are many graphs with geometric properties that are not shared by Cayley graphs of finitely generated groups. Thus, we aim to find a class of non-vertex-transitive graphs that are as close as possible to being quasi-isometric to finitely generated groups. We start our discussion by considering a class of graphs that satisfy a coarse notion of vertex-transitivity. To this end, we introduce the following definition. We say that a graph $X$ is coarsely transitive if there exists a constant $K \ge 1$ such that for any two vertices $x$ and $y$, there exists a $(K,K)$-quasi-isometry $X\to X$ which maps $x$ to within $K$-distance of $y$. One can see that every vertex-transitive graph is coarsely transitive, but conversely, one can construct a coarsely transitive graph which is not vertex-transitive by starting with any vertex-transitive graph and attaching a new vertex to it. Thus, one may ask what properties of vertex-transitive graphs pass to coarsely transitive graphs. With this in mind, we come to our first result. \\
 
\noindent {\bf Theorem~\ref{coarse separation corollary}  }
{\it Let $X$ be a coarsely transitive graph with two ends. Then $X$ is quasi-isometric to $\Z$.}\\
 
One may view this result as the coarsely transitive generalization of the well known fact, due to Freudenthal and Hopf, that all two-ended finitely generated groups are finite extensions of the integers. Freudenthal and Hopf also proved that finitely generated groups can only have zero, one, two, or infinitely many ends. Our next theorem shows that this phenomenon generalizes to coarsely transitive graphs.\\
 
\noindent {\bf Theorem~\ref{ends_coarse_transitive} }
{\it A coarsely transitive graph has either zero, one, two, or infinitely many ends.}\\

Although every locally finite vertex-transitive graph of polynomial growth is quasi-isometric to a finitely generated nilpotent group, we do find continuum many locally finite regular graphs with integral degree of polynomial growth, one or two ends, and finite asymptotic dimension, which are not quasi-isometric to any finitely generated group.\\

\noindent {\bf Theorem~\ref{main theorem}  }
{\it Given an infinite, locally finite, connected, vertex-transitive graph $X$, there exist continuum many pairwise non-quasi-isometric 3-regular graphs that have the same growth rate, number of ends, and asymptotic dimension as $X$.
 
In particular, for any infinite, finitely generated nilpotent group $G$, there exist continuum many pairwise non-quasi-isometric 3-regular graphs that have the same degree of polynomial growth, number of ends, and asymptotic dimension as $G$.  }\\

The proof of Theorem \ref{main theorem} proceeds by attaching line segments to our base Cayley graph $X$ along an infinite geodesic ray in the following way. After fixing a base point and a parameter $\alpha \in (0,1]$, we attach a segment of length $\lceil\log(n)^{\alpha} \rceil$ to the vertex on the ray at distance $n^2$ from the ray's endpoint. Calling this graph $X_\alpha$, we then demonstrate that the image of any quasi-isometric embedding of $X$ into $X_\alpha$ lies in a bounded neighborhood of $X \subset X_\alpha$. Since the attached segments along the ray grow without bound, it then follows that $X_\alpha$ and $X$ are not quasi-isometric. Moreover, the parameter $\alpha$ controls the growth rate of the attached line segments in such a way that $X_\alpha$ and $X_\beta$ are not quasi-isometric for distinct $\alpha,\beta$ in $(0, 1].$ On the other hand, since the attached segments are sparse and grow slowly in length, the graphs $X_\alpha$ share several large-scale geometric properties with $X$. 

\section{Notation and Basic Definitions}
For a metric space $X$, we use $d(x,y)$ to denote the distance between $x$ and $y$. We denote the $r$-ball about $x$ by $B_{X}(x,r)$ and the $r$-sphere about $x$ in $X$ by $S_{X}(x,r)$. When the metric space $X$ is clear from context, we simply write $B(x,r)$ and $S(x,r)$.

Let $f \colon (X, d_X) \to (Y, d_Y)$ be a map of metric spaces. We say that $f$ is an $(L,A)$\textbf{-quasi-isometric embedding} if there are constants $L \geq 1$, $A \geq 0$ such that for every $a,b \in X$,
$$
\frac{1}{L} d_X(a,b) - A \leq d_Y(f(a), f(b)) \leq L d_X(a,b) + A.
$$
An $(L,A)$-quasi-isometric embedding $f$ is an $(L,A)$-\textbf{quasi-isometry} if there is an $(L',A')$-quasi-isometric embedding $g : Y \to X$ such that $d_X(g\circ f,\text{Id}_X) < \infty$ and $d_Y(f\circ g,\text{Id}_Y) < \infty$, and we call $g$ a \textbf{quasi-inverse} of $f$. Equivalently, an $(L,A)$-quasi-isometric embedding $f$ is an $(L,A)$-quasi-isometry if it is \textbf{coarsely surjective}, that is, if there is a $C \ge 0$ such that the image of $f$ is $C$-dense in $Y$. A map $f \colon X \to Y$ is a \textbf{quasi-isometry} between $X$ and $Y$ if it is an $(L,A)$-quasi-isometry for some $L \ge 1, A \ge 0$. Two metric spaces $X$ and $Y$ are \textbf{quasi-isometric} if there exists a quasi-isometry between them.

A \textbf{graph} is a pair of sets $X = (V,E)$ where $E \subset V \times V$. We call $V$ the set of \textbf{vertices} and $E$ the set of \textbf{edges}. We denote the vertices of a graph $X$ as $V(X)$ and the edges of a graph as $E(X)$. Given an edge $\s{x,y}$, we call $x$ and $y$ the endpoints of $\s{x,y}$ and we say that $x$ and $y$ are adjacent. A \textbf{graph isomorphism} between graphs $X$ and $Y$ is a bijection $f : V(X) \to V(Y)$ such that $x$ and $y$ are adjacent in $V(X)$ if and only if $f(x)$ and $f(y)$ are adjacent in $V(Y)$. A \textbf{graph automorphism} of a graph $X$ is a graph isomorphism from $X$ to itself. A graph is \textbf{vertex-transitive} (or simply \textbf{transitive}) if its automorphism group acts transitively on its vertices. 

A graph is \textbf{connected} if any two vertices can be connected by a path. For any connected graph $X$, a natural metric is induced on the set of vertices by defining the distance between two vertices as the length of a shortest path between them. Since we are mainly interested in viewing graphs as metric spaces, we use the symbol $X$ to denote both the graph and the corresponding metric space. If $S \subset V(X)$, then the \textbf{subgraph of $X$ induced by $S$} is the graph whose vertex set is $S$ and whose edge set is the subset of edges in $E(X)$ that have both endpoints in $S$. We reuse the symbol $S$ to denote this induced subgraph, and we use $X \setminus S$ to denote the subgraph of $X$ induced by $V(X) \setminus S$. For the entirety of this note, we only consider graphs that are connected and unbounded as metric spaces.

Next we recall the definitions of some large-scale geometric properties of graphs. For a graph $X$ and subgraph $S$, let $U(X,S)$ denote the set of unbounded connected components of $X \setminus S$. Letting $X$ be a connected graph, we define the \textbf{number of ends of $X$} to be
\[
    e(X) = \sup\s{\abs{U(X,B)} : B \text{ is a bounded subgraph of } X}.
\]
Note that in particular, a graph has zero ends if and only if it is bounded. For graphs $X$ that are connected and unbounded, there is an equivalent definition in terms of rays. A \textbf{ray} in $X$ is a semi-infinite simple path; that is, it is an infinite sequence of vertices $v_0,v_1,\dots$ such that each vertex appears at most once in the sequence and every two consecutive vertices are adjacent. Two rays $r_1$ and $r_2$ are said to be equivalent if there is a ray $r_3$ that contains infinitely many of the vertices in each of $r_1$ and $r_2$. This defines an equivalence relation on the set of rays in $X$. Then the \textbf{ends} of $X$ are defined to be the equivalence classes of rays in $X$, and $e(X)$ is equal to the cardinality of the set of ends of $X$.

Let $X$ be a metric space, and let $n \ge 0$ be an integer. We say that asdim($X) \le n$ if for every $R \ge 1$ there is a uniformly bounded cover $\mathcal{U}$ of $X$ such that every ball in $X$ of radius $R$ intersects at most $n+1$ elements of $\mathcal{U}$ (here $\mathcal{U}$ is uniformly bounded if $\sup_{U \in \mathcal{U}} \text{diam}(U) < \infty$). Then the \textbf{asymptotic dimension} of $X$, denoted by asdim($X$), is the smallest integer $n \ge 0$ such that asdim($X) \le n$. If no such $n$ exists, we define asdim($X) = \infty$.

Let $f,g : \N \to \N$ be increasing functions. We write $f \preceq g$ if there is a $c \in \N$ such that $f(n) \le cg(cn+c)$ for all $n \in \N$. If $f \preceq g$ and $g \preceq f$, then we write $f \approx g$ and say that $f$ and $g$ are (asymptotically) equivalent. Note that $\approx$ defines an equivalence relation on the set of increasing functions $\N \to \N$. Suppose that $X$ is an unbounded, locally finite graph, and fix a vertex $x_0 \in X$. Let $f_{X,x_0} : \N \to \N$ be defined by $f_{X,x_0}(n) = \abs{B_X(x_0,n)}$. Observe that if $x_1$ is another vertex and $c = d(x_0,x_1)$, then $B(x_1,n) \subset B(x_0,n+c)$. So $f_{X,x_1}(n) \le f_{X,x_0}(n+c)$, which means $f_{X,x_1} \preceq f_{X,x_0}$. By symmetry we get $f_{X,x_0} \approx f_{X,x_1}$. Hence the equivalence class of $f_{X,x_0}$, which we define to be the \textbf{growth rate} of $X$, does not depend on the choice of $x_0$. Thus we can talk about the growth function $f_X$ of $X$, which is well-defined up to equivalence. In particular, two graphs $X$ and $Y$ are said to have the same growth rate if their growth functions $f_X$ and $f_Y$ are equivalent, that is, $f_X \approx f_Y$. If $f_X \preceq n^d$ for some integer $d \ge 0$, then we say that $X$ has a \textbf{polynomial growth rate}. In this case, the smallest $d$ for which $f_X \preceq n^d$ is called the \textbf{order of polynomial growth}.

Recall that growth rate, number of ends, and asymptotic dimension are invariant under quasi-isometry.

\section{Coarsely Transitive Graphs}

Let $K \ge 1$. A graph $X$ is said to be \textbf{$K$-coarsely transitive} if for any pair of vertices $x,y$ in $X$, there exists a $(K,K)$-quasi-isometry $f \colon X \to X$ such that $d(f(x),y) \le K$. A graph is said to be \textbf{coarsely transitive} if it is $K$-coarsely transitive for some $K \ge 1$. Note that all transitive graphs are coarsely transitive.

Suppose $X$ is $K$-coarsely transitive, and take vertices $x$ and $y$ in $X$. By definition, there is a $(K,K)$-quasi-isometry $f:X\to X$ with $d(f(x),y) \le K$. Define $f' \colon X \to X$ by $f'(x) = y$ and $f'(z)=f(z)$ for all $z \ne x$. Then $f'$ is a $(K,2K)$-quasi-isometry with $f'(x) = y$. Hence, we obtain the following lemma.

\begin{lemma}
A metric space $X$ is coarsely transitive if and only if there is some $K \ge 1$ such that for any pair of vertices $x,y$ in $X$, there is a $(K,K)$-quasi-isometry $f:X\to X$ with $f(x) = y$.
\end{lemma}

Moving forward, we take $K$-coarsely transitive to mean this equivalent condition.

First, we observe a basic obstruction to coarse transitivity. By an abuse of notation, we use $\mathbb{Z}^d$ to denote the Cayley graph of $\Z^d$ with respect to the standard symmetric generating set. Let $B_n = B(0,n)$ be the subgraph of $\Z$ induced by the vertex set $\s{k \in \Z : |k| \le n}$.

\begin{lemma}
Let $L\ge 1$ and $A \ge 0$ be given. For all sufficiently large $n$, if $f : B_n \to \mathbb{Z}$ is an $(L,A)$-quasi-isometric embedding, then either $f(-n) < f(0) < f(n)$ or $f(n) < f(0) < f(-n)$.
\end{lemma}

\begin{proof}
Letting $n > L^2+2LA$, we have that for each $k=-n,-n+1,\dots,0$,
\[
    d(f(n),f(k)) \ge 
    \tfrac{1}{L}(n-k) - A \ge
    \tfrac{1}{L}n - A >
    L+A.
\]
Similarly, $d(f(-n),f(k)) > L+A$ for each $k=0,1,\dots,n$. On the other hand, for each $k=-n,-n+1,\dots,n-1$, we have
\[
    d(f(k),f(k+1)) \le L+A.
\]
First, suppose that $f(-n) < f(0)$. If $m \le f(-n)$, then
\[
    d(m,f(0)) \ge
    d(f(-n),f(0)) >
    L+A.
\]
Since $d(f(1),f(0)) \le L+A$, we must have $f(-n) < f(1)$. By induction we get $f(-n) < f(n)$. Now, if $f(n) < f(0)$, then by a similar argument, we would get $f(n) < f(-n)$ which contradicts $f(-n) < f(n)$. Thus, $f(0) < f(n)$, and we have $f(-n) < f(0) < f(n)$, as desired.

Now suppose that $f(0) < f(-n)$. Then by a symmetric argument, we get $f(n) < f(0) < f(-n)$.
\end{proof}

Let $T_n$ be the subgraph of $\Z^2$ induced by the vertex set $\s{(k,0) : |k| \le n} \cup \s{(0,k) : 0 \le k \le n}$. Then $T_n$ can be thought of as a tripod with legs of length $n$.

\begin{proposition} \label{obstruction to coarse transitivity}
Let $L \ge 1$ and $A \ge 0$ be given. Then for all sufficiently large $n$, there does not exist an $(L,A)$-quasi-isometric embedding $T_n \to \Z$.
\end{proposition}

\begin{proof}
The union of any two legs of $T_n$ is isometric to $B_n$. Therefore, $T_n$ contains three distinct subgraphs, $S_1,S_2$, and $S_3$, each of which is isometric to $B_n$. Suppose for contradiction that there is an $(L,A)$-quasi-isometric embedding $f : T_n \to \Z$. Then for $i=1,2,3$, the restriction $f\vert_{S_i}$ is an $(L,A)$-quasi-isometric embedding of $S_i \cong B_n$ into $\Z$. Using the previous lemma, we may assume without loss of generality that $f(-n,0) < f(0,0) < f(n,0)$. Then $f(-n,0) < f(0,0)$ implies $f(0,0) < f(0,n)$, and $f(0,0) < f(n,0)$ implies $f(0,n) < f(0,0)$. This is impossible, so no such $f$ exists.
\end{proof}

Hence, a graph which has arbitrarily large parts which coarsely look like $T_n$ and $B_n$ cannot be coarsely transitive. For example, the subgraph of $\Z^2$ induced by the vertex set $\s{(x,y) : \abs y \le \abs x}$ is not coarsely transitive, because for each $K \ge 1$, there is an $n \gg k$ such that no $(K,K)$-quasi-isometry which maps $(n,0)$ to $(0,0)$ exists.

It is known that connected transitive graphs which are unbounded have either one, two, or infinitely many ends \cite{DJM}, and moreover that two-ended transitive graphs are quasi-isometric to $\Z$ \cite{MR}. We show that these two properties extend to coarsely transitive graphs.

\begin{proposition} \label{coarse separation}
Let $X$ be a coarsely transitive graph with at least two ends, and let $B_0 = B(x_0,r)$ be a ball with $\abs{U(X,B_0)} \ge 2$. Then there is an $R > 0$ such that for any ball $B$ of radius $R$, we have $\abs{U(X,B)} \ge \abs{U(X,B_0)}$.
\end{proposition}

\begin{proof}
Suppose $X$ is $K$-coarsely transitive. Set $R = Kr + 3K^2$, and let $B_1 = B(x_1,R)$ where $x_1$ is arbitrary. Since $X$ is $K$-coarsely transitive, there is a $(K,K)$-quasi-isometry $f \colon X \to X$ with $f(x_1) = x_0$. We will show that $f$ induces a surjection $U(X,B_1) \to U(X,B_0)$. First, we observe that
\begin{align}
    y \notin B_1 \implies
    d(x_0,f(y)) =
    d(f(x_1),f(y)) \ge 
    \tfrac{1}{K}d(x_1,y)-K >
    \tfrac{1}{K}R - K =
    r + 2K.
\end{align}
That is, $f$ maps the complement of $B_1$ to the complement of $B(x_0,r+2K)$. Let $C \in U(X,B_1)$, and pick any $y \in C$. Then $f(y) \notin B_0$. Hence, let $D$ be the connected component of $X \setminus B_0$ which contains $f(y)$. We now show that $f(C) \subset D$. Let $z \in C$ be adjacent to $y$. Then
\[
    d(f(y),f(z)) \le 
    Kd(y,z) + K =
    2K.
\]
Let $\gamma$ be a path of minimal length in $X$ between $f(y)$ and $f(z)$. If $f(z) \notin D$, then there must exist a point $w \in \gamma \cap B_0$ (see Figure \ref{fig:d1}).

\begin{figure}[H]
    \centering
    \includegraphics[width=65mm]{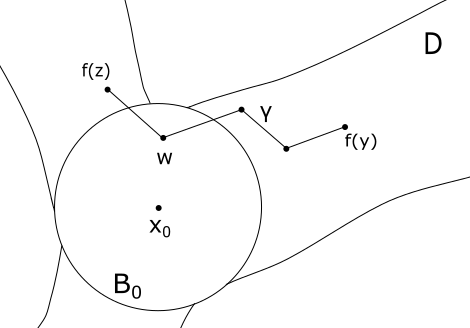}
    \caption{}
    \label{fig:d1}
\end{figure}

In this case,
\[
    d(x_0,f(y)) \le
    d(x_0,w) + d(w,f(y)) \le
    r + \text{length}(\gamma) =
    r + d(f(y),f(z)) \le
    r + 2K,
\]
contrary to (1). Therefore, $f(z) \in D$. Since $C$ is connected, it follows that $f(C) \subset D$. Since $C$ is unbounded and $f$ is a quasi-isometry, $D$ must be unbounded. Therefore, $D \in U(X,B_0)$, and we let $F(C) = D$. Thus we get a well-defined map $F : U(X,B_1) \to U(X,B_0)$ which we will show is surjective. Let $D \in U(X,B_0)$, and let $g$ be a quasi-inverse of $f$. Without loss of generality we may assume that $g(x_0) = x_1$. Like before, take a ball $B_2 = B(x_0,R')$ of sufficiently large radius $R'$ such that
\[
    x \notin B_2 \implies d(x_1,g(x)) > R+2K.
\]
Then by similar reasoning used before, $g$ maps each element of $U(X,B_2)$ into an element of $U(X,B_1)$. Since $D$ is an unbounded component of $X \setminus B_0$, and $B_2$ is just a bounded neighborhood of $B_0$, there must be a $D' \in U(X,B_2)$ with $D' \subset D$. (In fact, $D'$ is just the subset of points in $D$ whose distance from $x_0$ exceeds $R'$). Then like before, $g(D') \subset C$ for some $C \in U(X,B_1)$. Since $f$ and $g$ are quasi-inverses and $D'$ is unbounded, there is a point in $C$ that $f$ maps into $D' \subset D$. Hence $f(C) \subset D$, and therefore $F(C) = D$. Thus, $F$ is surjective, and $\abs{U(X,B_1)} \ge \abs{U(X,B_0)}$. Since $B_1$ was an arbitrary ball of radius $R$, we are done.
\end{proof}

When $X$ is two-ended, we get the following corollary.

\begin{corollary} \label{coarse separation corollary}
If $X$ is a coarsely transitive graph with two ends, then there is an $R>0$ such that every ball $B$ of radius $R$ satisfies $\abs{U(X,B)} = 2$.
\end{corollary}

\begin{proof}
Since $e(X) = 2$, there is a ball $B_0$ with $\abs{U(X,B_0)} = 2$. Thus, Proposition \ref{coarse separation} implies that there is an $R > 0$ such that every ball $B$ of radius $R$ has $\abs{U(X,B)} \ge \abs{U(X,B_0)} = 2$. On the other hand, since $e(X) = 2$, every such ball $B$ has $\abs{U(X,B)} \le 2$.
\end{proof}

We first show that a two-ended, coarsely transitive graph X is quasi-isometric to $\Z$. By Corollary \ref{coarse separation corollary}, any ball of sufficiently large radius (independent of the center point) will roughly separate $X$ into two unbounded components. Thus, we may construct a bi-infinite, pairwise-disjoint sequence of such balls, and this sequence of balls will look like the integers when viewing $X$ from afar. Indeed, the map which sends the integers to the centers of the balls will be the desired quasi-isometry.

\begin{theorem}
Let $X$ be a coarsely transitive graph. If $e(X) = 2$, then $X$ is quasi-isometric to $\Z$.
\end{theorem}

\begin{proof}
Using Corollary \ref{coarse separation corollary}, let $r > 0$ such that any ball $B$ of radius $r$ satisfies $\abs{U(X,B)} = 2$. Fix a vertex $x_0$ and let $B_0 = B(x_0,r)$. Let $P_0$ and $N_0$ denote the two elements of $U(X,B_0)$. Pick a vertex $x_1 \in P_0$ with $d(x_0,x_1) = 2r+1$. Let $B_1 = B(x_1,r)$ and note that $B_1 \subset P_0$. Then $N_0 \cup B_0$ is an unbounded connected subgraph of $X \setminus B_1$ and thus must be contained in one of the two elements of $U(X,B_1)$. Let $N_1 \in U(X,B_1)$ denote the component containing $N_0 \cup B_0$, and let $P_1 \in U(X,B_1)$ denote the other element. Since $N_0 \cup B_0 \subset N_1$, it follows that $P_1 \subset P_0$. Then pick a vertex $x_2 \in P_1$ with $d(x_1,x_2) = 2r+1$, and similarly define $B_2,P_2$, and $N_2$ (see Figure \ref{fig:d2}).

\begin{figure}[H]
    \centering
    \includegraphics[width=100mm]{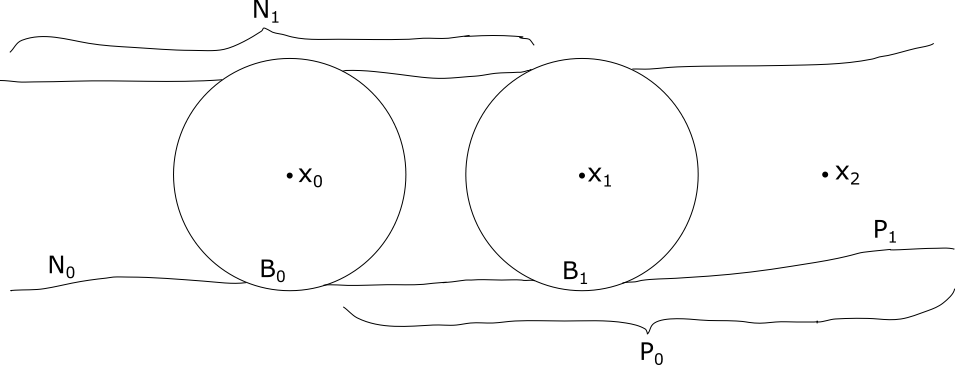}
    \caption{}
    \label{fig:d2}
\end{figure}

We continue this process, as well as a symmetric process in the direction of $N_0$ instead of $P_0$, to construct $x_k$, $B_k$, $P_k$, and $N_k$ for all $k \in \Z$, such that
\begin{itemize}
    \item $d(x_k,x_{k+1}) = 2r+1$,
    \item $B_{k+1} \cup P_{k+1} \subset P_k$,
    \item $N_k \cup B_k \subset N_{k+1}$.
\end{itemize}
Now, consider the map $\Z \to X$ given by $k \mapsto x_k$. We show that this is a bi-Lipschitz embedding. Let $m,n \in Z$ with $m < n$. By the triangle inequality,
\begin{align*}
    d(x_m,x_n) &\le
    d(x_m,x_{m+1}) + \dots + d(x_{n-1},x_n) \\&=
    (2r+1) + \dots + (2r+1) \\&=
    (2r+1)(n-m).
\end{align*}

For the other inequality, let $\gamma$ be a path between $x_m$ and $x_n$ of minimal length. Since $B_{k+1} \cup P_{k+1} \subset P_k$ and $N_k \cup B_k \subset N_{k+1}$, we have $x_n \in P_j$ and $x_m \in N_j$ for each $j= m+1,\dots,n-1$. Since $\gamma$ is a path between $N_j$ and $P_j$, it must intersect $B_j$. Hence, for each $j=m,\dots,n-1$, $\gamma$ contains a sub-path $\eta_j$ between $B_j$ and $B_{j+1}$, and since $\gamma$ has minimal length, the $\eta_j$ have non-overlapping edges (see Figure \ref{fig:d3}).
\begin{figure}[H]
    \centering
    \includegraphics[width=160mm]{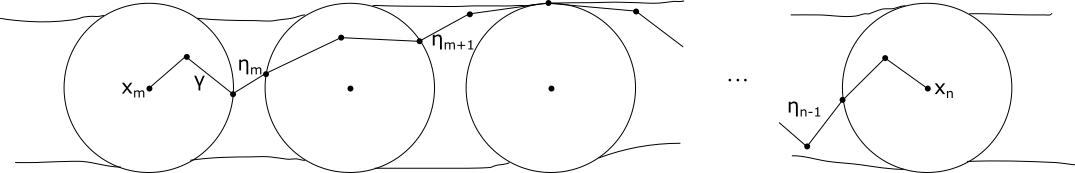}
    \caption{}
    \label{fig:d3}
\end{figure}
Since the $B_k$ are pairwise disjoint, length($\eta_j) \ge 1$. Hence,
\[
    d(x_m,x_n) =
    \text{length}(\gamma) \ge
    \sum_{j=m}^{n-1}\text{length}(\eta_j) \ge
    \sum_{j=m}^{n-1} 1 =
    n-m.
\]

It remains to show that our given map is coarsely surjective. For each $k \in \Z$, let $\gamma_k$ be a path between $x_k$ and $x_{k+1}$ of length $d(x_k,x_{k+1}) = 2r+1$, and define $L$ to be the union of the $\gamma_k$.

Let $x \in X$ be arbitrary and let $B = B(x,r)$. Let $\delta = d(x,x_0)$, and let $N \ge 3r+\delta+1$ be such that $x_N \notin B$. Let $C$ be the component of $X \setminus B$ which contains $x_N$. Suppose for contradiction that $\s{x_k}_{k\ge N}$ is not entirely contained in $C$, and let $n > N$ be the smallest index with $x_n \notin C$. Thus, $x_{n-1} \in C$, and we let $\gamma$ be a path of minimal length between $x_{n-1}$ and $x_n$. Since $x_n \notin C$, $\gamma$ must intersect $B$ at a point, say, $w$. Then
\[
    d(x_n,x_0) \le 
    d(x_n,w) + d(w,x) + d(x,x_0) \le
    (2r+1) + r + \delta =
    3r + \delta + 1,
\]
and
\[
    d(x_n,x_0) \ge
    n-0 >
    N \ge
    3r+\delta+1,
\]
which is a contradiction. Thus, $\s{x_k}_{k\ge N} \subset C$ which means that $C \in U(X,B)$. Similarly, we may assume (by taking possibly larger $N$) that $\s{x_k}_{k \le -N}$ is contained in an element of $U(X,B)$. Let $m > N$ be such that $B_m \subset C$. If $\s{x_k}_{k \le -N} \subset C$, then we would have $\abs{U(X,B \cup B_m)} \ge 3$ (see Figure \ref{fig:d4}). 

\begin{figure}[H]
    \centering
    \includegraphics[width=110mm]{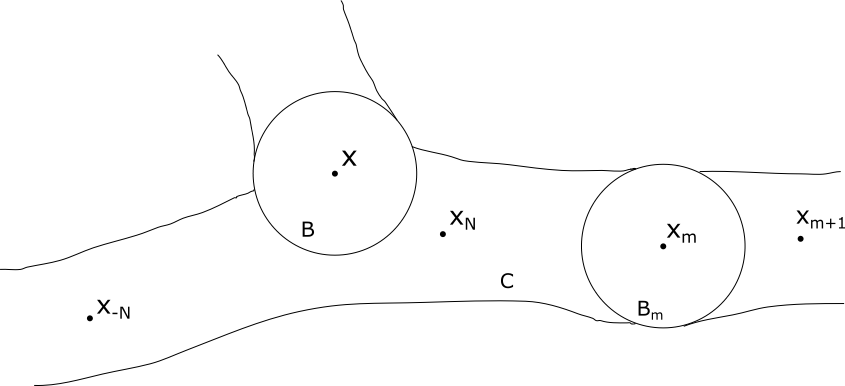}
    \caption{}
    \label{fig:d4}
\end{figure}

This would contradict $e(X) = 2$; therefore $\s{x_k}_{k \le -N}$ must be contained in the other unbounded component of $X \setminus B$. Since $x_{-N}$ and $x_N$ are in separate components of $X \setminus B$, every path between them must intersect $B$. In particular, $L$ must intersect $B$ at some point $v$. Then $v \in \gamma_k$ for some $k \in \Z$, and therefore, 
\[
    d(x_k,x) \le
    d(x_k,v) + d(v,x) \le
    (2r+1) + r =
    3r + 1.
\]
Since $x$ was chosen arbitrarily, coarse surjectivity follows.
\end{proof}
We finish this section with the coarsely transitive generalization of the classification of ends of transitive graphs.
\begin{theorem}\label{ends_coarse_transitive}
A coarsely transitive graph has either zero, one, two, or infinitely many ends.
\end{theorem}

\begin{proof}
Let $X$ be a coarsely transitive graph with more than two ends. By Proposition \ref{coarse separation}, there is an $r > 0$ such that any ball $B$ of radius $r$ satisfies $\abs{U(X,B)} \ge 3$. Fix a vertex $x_0 \in X$, and let $B_0 = B(x_0,r)$. Then $U(X,B_0)$ has at least three elements, which we call $U_0,V_0,W_0$. Pick a vertex $x_1 \in W_0$ with $d(x_0,x_1) = 2r+1$. Let $B_1 = B(x_1,r)$ and note that $B_1 \subset W_0$. Then $B_0 \cup U_0 \cup V_0$ is an unbounded connected subgraph of $X \setminus B_1$, and hence is contained in an element, say $U_1$, of $U(X,B_1)$. Let $V_1$ and $W_1$ denote two other elements of $U(X,B_1)$. Then $U_0,V_0,V_1,W_1 \in U(X,B_0 \cup B_1)$ which implies that $\abs{U(X,B_0\cup B_1)} \ge 4$ (Figure \ref{fig:d5}).
\begin{figure}[H]
    \centering
    \includegraphics[width=100mm]{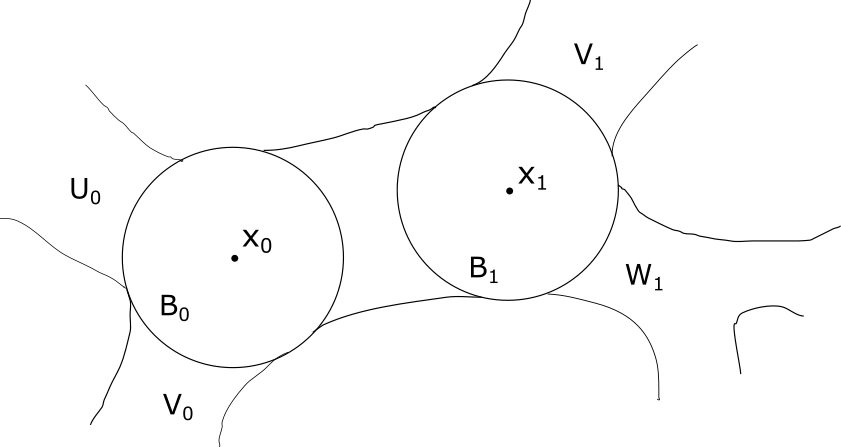}
    \caption{}
    \label{fig:d5}
\end{figure}
Continuing in this way, we pick a vertex $x_2 \in W_1$, and consider $B_2 = B(x_2,r)$. We then find that $\abs{U(X,B_0\cup B_1\cup B_2)} \ge 6$. Hence, we construct pairwise-disjoint balls $B_k$ for all $k \in \N$ so that for any $M > 0$, there is an $n \in \N$ for which $\abs{U(X,\bigcup_{k=0}^n B_k)} > 
M$. Thus, $e(X) = \infty$.
\end{proof}

\section{Quasi-isometry classes of graphs of polynomial growth}

Our goal for this section is to show that given an infinite, locally finite, connected, transitive graph $X$, there exist continuum many 3-regular graphs that are pairwise non-quasi-isometric and yet share several large-scale geometric properties. We define a \textbf{geodesic} $P$ in $X$ to be a bi-infinite path such that for any two vertices $x,y$ on $P$, $P$ contains a shortest-length path between $x$ and $y$. In other words, $d(x,y) = d_P(x,y)$ where $d$ is the path metric on $X$ and $d_P$ is the restriction of $d$ to $P$. By Theorem 4.1 in \cite{W}, every vertex in $X$ lies on a geodesic. Thus, take any geodesic $P$ in $X$ and label its vertices by $\s{x_n}_{n\in\Z}$ such that $d(x_m,x_n) = |m-n|$ for all $m,n \in \Z$. We construct a family of graphs from $X$ as follows. For each $0 < \alpha \le 1$, define $g_\alpha : (0,\infty) \to \N$ by $g_\alpha(x) = \lceil \log(x)^\alpha \rceil$ and note that $g_\alpha(n) \le n$ for all $n \in \N$. For each positive integer $n$, let $S_n^\alpha$ be the subgraph of $\Z$ induced by the vertex set $\s{k : 0 \le k \le g_\alpha(n)}$. Define
\[
    X_\alpha \defeq \parens{X \sqcup \bigsqcup_{n>0} S_n^\alpha} / \sim
\]
where for each positive integer $n$, we identify $x_{n^2} \in X$ with $0 \in S_n^\alpha$. From now on, we denote the vertex $k \in S_n^\alpha \subset X_\alpha$ by $k_n$. For example, $0_n$ and $x_{n^2}$ denote the same vertex in $X_\alpha$, and $g_\alpha(n)_n$ denotes the ``tip" of the segment $S_n^\alpha$ in $X_\alpha$. To reduce clutter, we define $t_n^\alpha = g_\alpha(n)_n$ for all $n > 0$ (see Figure \ref{fig:d6}). 

\begin{figure}[H]
    \centering
    \includegraphics[width=100mm]{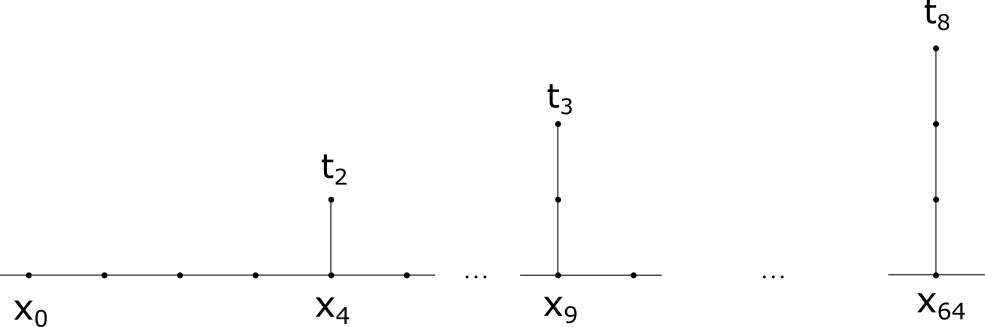}
    \caption{The graph $X_\alpha$ when $X = \Z$ and $\alpha = 1$}
    \label{fig:d6}
\end{figure}

For the remainder of this section, $X$ will denote an infinite, locally finite, connected, transitive graph, and if $0 < \alpha \le 1$, then $X_\alpha$ will denote the graph we constructed from $X$ above.

\begin{proposition}\label{invariants_same}
The graphs $X$ and $X_\alpha$ have the same number of ends, asymptotic dimension, and growth rates. 
\end{proposition}
\begin{proof}
Recall that the ends of a graph are given by equivalence classes of rays (semi-infinite paths with no self-intersection). Let $r$ be an arbitrary ray in $X_\alpha$. Since rays are infinite and do not have repeating vertices, $r$ can intersect $S_n^\alpha$ only at $0_n = x_{n^2} \in X$. Hence, $r$ is a ray in $X \subset X_\alpha$. Hence, all rays in $X_\alpha$ are just rays in $X$, and thus, the ends of $X_\alpha$ are identified with the ends of $X$. Therefore, $e(X_\alpha) = e(X)$.

Let $P$ be the geodesic in $X$ with respect to which $X_\alpha$ is defined, and consider the subgraph of $X_\alpha$ defined by
\[
Y \defeq P \cup \bigcup_{n > 0} S_n^\alpha.
\]
For example, if $X = \Z$, then $Y$ is all of $X_\alpha$. We claim that $Y$ isometrically embeds into the 3-regular tree $T$. Indeed, let $\gamma_0$ be a geodesic in $T$ with vertices $\s{v_k}_{k\in\Z}$, and map $P$ isometrically onto $\gamma_0$. Then for each $n \in \N$, since $T$ is a 3-regular tree, we may take a (necessarily geodesic) ray $\gamma_n$, which emanates from $v_n$ and does not intersect $\gamma_0$ elsewhere. Note that since $T$ is a tree, the rays $\gamma_n$ are pairwise disjoint. Then we may isometrically embed $S_n^\alpha$ into $\gamma_n$ with the condition that $0 \in S_n^\alpha$ is mapped to $v_n$. The result is an isometric embedding of $Y$ into $T$. Since trees have asymptotic dimension 1, we have $\text{asdim}(Y) \le 1$. Then
\[
    \text{asdim}(X_\alpha) = 
    \text{asdim}(X \cup Y) \le
    \max\s{\text{asdim}(X),\text{asdim}(Y)} =
    \text{asdim}(X).
\]
On the other hand, $\text{asdim}(X) \le \text{asdim}(X_\alpha)$ because $X \subset X_\alpha$. Thus $\text{asdim}(X_\alpha) = \text{asdim}(X)$.

Lastly, we show that $X$ and $X_\alpha$ have the same rate of growth. First, note that since $X \subset X_\alpha$, we have $\abs{B_X(x_0,n)} \le \abs{B_{X_\alpha}(x_0,n)}$. Next, we show that $S_{X_\alpha}(x_0,n)$ has at most one more element than $S_X(x_0,n)$, in which case, 
\begin{align*}
    \abs{B_{X_\alpha}(x_0,n)} &= 
    1 + \sum_{i=1}^n \abs{S_{X_\alpha}(x_0,i)} \\&\le 
    1 + \sum_{i=1}^n \parens{\abs{S_{X}(x_0,i)} + 1} \\&=
    \abs{B_X(x_0,n)} + n \\&\le
    2\abs{B_X(x_0,n)},
\end{align*}
where $n \le \abs{B_X(x_0,n)}$ because $\s{x_1,\dots,x_n} \subset B_X(x_0,n)$. First note that $d(x_0,j_m) < d(x_0,k_m)$ for all $m > 0$ and $j < k \le g_\alpha(m)$. Moreover,
\[
   d(x_0,t_m^\alpha) =
   m^2 + g_\alpha(m) \le
   m^2+m =
   m(m+1) <
   (m+1)^2 =
   d(x_0,0_{m+1}).
\]
Since $\abs{S_{X_\alpha}(x_0,1)} = \abs{S_X(x_0,1)}$, we let $n > 1$, and let $m^2$ be the largest square such that $m^2 \le n$. Note that the intersection of $S_{X_\alpha}(x_0,n)$ with $X \subset X_\alpha$ is $S_X(x_0,n)$. Since $d(x_0,0_m) \le n$, it follows from the above observations that $S_{X_\alpha}(x_0,n)$ does not intersect $S_k^\alpha$ for any $k < m$. Also, since $d(x_0,0_{m+1}) > n$, $S_{X_\alpha}(x_0,n)$ does not intersect $S_k^\alpha$ for any $k > m$. Finally, $S_{X_\alpha}(x_0,n)$ intersects $S_m^\alpha$ once if $m^2 \le n \le m^2 + g_\alpha(m)$, and otherwise, the intersection is empty. Hence, $\abs{S_{X_\alpha}(x_0,n)} \le \abs{S_{X}(x_0,n)} + 1$. Thus,
\[
    \abs{B_X(x_0,n)} \le 
    \abs{B_{X_\alpha}(x_0,n)} \le
    2\abs{B_X(x_0,n)},
\]
which implies that $X$ and $X_\alpha$ have the same rate of growth.
\end{proof}

While $X$ and $X_\alpha$ share some large-scale geometric properties, it turns out by the following proposition that they are not quasi-isometric to each other.

\begin{proposition}\label{base prop}
If $f : X \to X_\alpha$ is a quasi-isometric embedding, then $\sup_{x\in X}d(f(x),X) < \infty$.
\end{proposition}

\begin{proof}
Suppose $f : X \to X_\alpha$ is an $(L,A)$-quasi-isometric embedding. Assume for contradiction that for some $x \in X$, 
\[
    d(f(x),X) > L^3 + 2L^2A + A.
\]
Let $P = \s{x_n}_{n\in\Z}$ be the geodesic with respect to which $X_\alpha$ is defined. Since $X$ is transitive, we may assume without loss of generality that $x=x_0$. Then $f(x_0) \in S_k^\alpha$ for some $k$, but since $S_k^\alpha$ is bounded, there must be an $m < 0$ such that $f(x_m) \notin S_k^\alpha$ and $f(x_{m+1}) \in S_k^\alpha$, and an $n > 0$ such that $f(x_{n-1}) \in S_k^\alpha$ and $f(x_n) \notin S_k^\alpha$ (see Figure \ref{fig:d7}). 

\begin{figure}[H]
    \centering
    \includegraphics[width = 75mm]{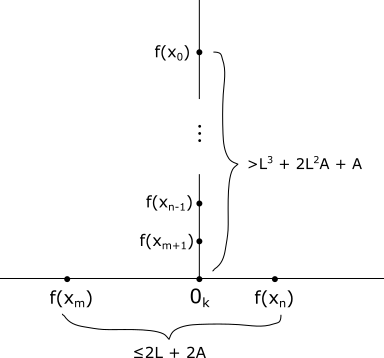}
    \caption{}
    \label{fig:d7}
\end{figure}

We then have that
\begin{align*}
    d(f(x_m),f(x_n)) &\le
    d(f(x_m),0_k) + d(0_k,f(x_n)) \\&\le
    d(f(x_m),f(x_{m+1})) + d(f(x_{n-1}),f(x_n)) \\&\le
    (L+A) + (L+A) \\&=
    2L+2A.
\end{align*}
For all $y \in X$ with $d(x_0,y) \le L^2 + 2LA$, it follows that
\[
    d(f(x_0),f(y)) \le L(L^2+2LA) + A = L^3 + 2L^2A + A.
\]
Since $f(x_m),f(x_n) \notin S_k^\alpha$ and $d(f(x_0),X) > L^3 + 2L^2A + A$, it follows that $|m|,n > L^2 + 2LA$, and hence, $n-m > 2L^2 + 4LA$. Thus,
\[
    d(f(x_m),f(x_n)) > \frac{1}{L}(2L^2+4LA) - A = 
    2L+3A,
\]
which is a contradiction. Therefore, $\sup_{x\in X}d(f(x),X) \le L^3 + 2L^2A + A$.
\end{proof}

\begin{corollary}
$X$ and $X_\alpha$ are not quasi-isometric for any $0 < \alpha \le 1$.
\end{corollary}

\begin{proof}
By Proposition \ref{base prop}, the image of any quasi-isometric embedding $X \to X_\alpha$ lies in a bounded neighborhood of the base graph $X \subset X_\alpha$. But in $X_\alpha$, since the lengths of the segments $S_n^\alpha$ increase without bound, the distance from the tips $t_n^\alpha$ of those segments to the base $X \subset X_\alpha$ grow arbitrarily large. Thus, any quasi-isometric embedding $X \to X_\alpha$ cannot be coarsely surjective.
\end{proof}

Our next result states that furthermore, $X_\alpha$ and $X_\beta$ are not quasi-isometric if $\alpha \ne \beta$. The proof uses a similar argument. We show that any quasi-isometric embedding $X_\alpha \to X_\beta$, for $\alpha <\beta$, fails to be coarsely surjective. If $f : X_\alpha \to X_\beta$ is a quasi-isometric embedding, then Proposition \ref{base prop} implies that $f$ maps the base graph $X \subset X_\alpha$ to a neighborhood of the base graph $X \subset X_\beta$. Hence, the segments $S_n^\alpha$ in $X_\alpha$ must be coarsely mapped to the segments $S_n^\beta$ in $X_\beta$. Now, the lengths of $S_n^\beta$ grow faster than the lengths of $S_n^\alpha$, but since the consecutive distances between the $S_n^\alpha$ grow quadratically, and $f$ distorts distances up to a fixed linear factor, we will see that the distances between the tips $t_n^\beta$ in $X_\beta$ and $f(X_\alpha)$ grow arbitrarily large.

For the proof, we first need a lemma.

\begin{lemma}\label{limit_lemma}
Let $0 <\alpha < \beta$. If $T \colon \R \to \R$ is affine and $p \in \mathbb{R}[x]$ is a polynomial such that $p > 0$ on $(0,\infty)$, then 
\[
    \lim_{x\to\infty}
    \frac{T(g_\alpha(p(x)))}{g_\beta(x)} =  0.
\]
\end{lemma}

\begin{proof}
Firstly, $\log(x)^{\alpha-\beta} \to 0$ as $x\to\infty$, and thus, $g_\alpha(x)/g_\beta(x) \to 0$ as $x\to\infty$. Let $d = \deg p$. Since $\lim_{x\to\infty}\log(p(x))/\log(x) \to d$, we have that $g_\alpha(p(x))/g_\alpha(x) \to d^\alpha$ as $x\to\infty$. Hence,
\[
    \lim_{x\to\infty} \frac{g_\alpha(p(x))}{g_\beta(x)} =
    \lim_{x\to\infty} \parens{\frac{g_\alpha(p(x))}{g_\alpha(x)} \cdot \frac{g_\alpha(x)}{g_\beta(x)}} =
    \lim_{x\to\infty} \frac{g_\alpha(p(x))}{g_\alpha(x)} \cdot \lim_{x\to\infty} \frac{g_\alpha(x)}{g_\beta(x)} =
    d^\alpha \cdot 0 =
    0.
\]
Since $T$ is affine and $g_\beta(x) \to \infty$ as $x\to\infty$, we obtain the desired equality.
\end{proof}

\begin{proposition}\label{distinct_quasi}
For $0 < \alpha, \beta \le 1$, if $\alpha \ne \beta$, then $X_\alpha$ and $X_\beta$ are not quasi-isometric.
\end{proposition}

\begin{proof}
Assume $\alpha < \beta$, and let $f : X_\alpha \to X_\beta$ be an $(L,A)$-quasi-isometric embedding. Since $f$ is arbitrary, we are done if we show that $f$ is not coarsely surjective. Recall that for $0 < \gamma \le 1$, we denote the ``tip" of the segment $S_n^\gamma \subset X_\gamma$ by $t_n^\gamma$. Let $M > 0$. We show that if $n$ is sufficiently large, then $d(f(X_\alpha),t_n^\beta) > M$.

Since $f\vert_X$ is a quasi-isometric embedding $X \to X_\beta$, Proposition \ref{base prop} implies there is a $D > 0$ such that $\sup_{x \in X}d(f(x),X) \le D$. By Lemma \ref{limit_lemma}, we have
\[
    \lim_{x\to\infty} \frac{Lg_\alpha(L(x+2x^2+A)) + A}{g_\beta(x)} \to 0.
\]
In particular, there is an $N > 0$ such that for all $x \ge N$, we have that
\[
    g_\beta(x) > Lg_\alpha(L(x+2x^2+A)) + A + (M + D).
\]
Fix an integer $n \ge N$ large enough so that $d(f(x_0),0_n) \le 2n^2$, and set $R  = L(n + d(f(x_0),0_n) + A)$. Then for all $x \in X_\alpha$ with $d(x_0,x) > R$, we have that
\begin{align*}
    d(f(x_0),f(x)) &\ge 
    \frac{1}{L}d(x_0,x) - A \\&>
    n + d(f(x_0),0_n) \\&\ge
    g_\beta(n) + d(f(x_0),0_n) \\&\ge
    d(f(x_0),t_n^\beta).
\end{align*}
We now claim that $d(f(x),t_n^\beta) > M$. Since $g_\beta(n)>M+D$, if $d(f(x),t_n^\beta) \le M$ then $f(x) \in S_n^\beta$, and in particular, $f(x)$ is on every path between $f(x_0)$ and $t_n^\beta$. However, this would imply $d(f(x_0),f(x)) \le d(f(x_0),t_n^\beta)$. 

It remains to show that $d(f(x),t_n^\beta) > M$ for all $x \in X_\alpha$ with $d(x_0,x) \le R$. Observe that
\[
    B_{X_\alpha}(x_0,R) \subset 
    B_X(x_0,R) \cup \bigcup_{i=1}^{\lfloor \sqrt{R} \rfloor} S_i^\alpha.
\]
If $x \in B_X(x_0,R)$, then we have $x \in X$ and thus, $d(f(x),X) \le D$. That implies $d(f(x),t_n^\beta) > M$ because $g_\beta(n) > M+D$. On the other hand, if $x \in S_k^\alpha$, where $k^2 \le R$, then
\begin{align*}
    d(f(0_k),f(x)) &\le
    Ld(0_k,x) + A \\&\le
    Lg_\alpha(k) + A \\&\le
    Lg_\alpha(R) + A \\&\le
    Lg_\alpha(L(n+2n^2+A)) + A.
\end{align*}
Therefore,
\begin{align*}
    d(f(x),t_n^\beta) &\ge 
    d(f(0_k),t_n^\beta) - d(f(0_k),f(x)) \\&\ge
    (g_\beta(n) - D) - (Lg_\alpha(L(n+2n^2+A)) + A) \\&>
    M.
\end{align*}
Thus, $d(f(x),t_n^\beta) > M$ for all $x \in X_\alpha$ with $d(x_0,x) \le R$. We have now shown that $d(f(x),t_n^\beta) > M$ for all $x \in X_\alpha$. Since $M$ can be arbitrarily large, $f$ is not coarsely surjective.
\end{proof}

\begin{theorem} \label{main theorem}
Given an infinite, locally finite, connected, transitive graph $X$, there exist continuum many pairwise non-quasi-isometric 3-regular graphs that have the same growth rate, number of ends, and asymptotic dimension as $X$.

In particular, given any infinite, finitely generated nilpotent group $G$, there exist continuum many pairwise non-quasi-isometric $3$-regular graphs that have the same degree of polynomial degree of growth, number of ends, and asymptotic dimension as $G$. 
\end{theorem}

\begin{proof}
By Proposition \ref{invariants_same}, each element of the set $\s{X_\alpha : 0 < \alpha \le 1}$ has the same number of ends, asymptotic dimension, and growth rate as $X$. By Proposition \ref{distinct_quasi}, they are all in distinct quasi-isometry classes. Finally, by Theorem 19 in \cite{BRSV}, every graph is quasi-isometric to a 3-regular graph.
\end{proof}

\section{Further Questions}
We finish this article by discussing some open questions concerning the geometry of coarsely transitive graphs. We start with the following question.\\

\noindent  {\bf Question 1} Does there exist a coarsely transitive graph that is not quasi-isometric to a transitive graph? If so, can we ensure that it is locally finite?\\

From the definition of vertex-transitivity, we have that the automorphism group of a transitive graph is always nontrivial. A similar property holds for coarsely transitive graphs in that the group of quasi-isometries is always nontrivial. Thus, if one could construct a coarsely transitive graph where every graph in its quasi-isometry class has a trivial automorphism group, one would have an answer to the above question. In particular such a graph would have coarse symmetries but no actual symmetries. 

Another question one may consider is whether there are examples of polynomially growing graphs, which fail to be quasi-isometric to any finitely generated group, but have more symmetries than the ones found in Theorem \ref{main theorem}. After all, one can see that by Proposition \ref{obstruction to coarse transitivity}, the constructed graphs $X_\alpha$ are not even coarsely transitive. Therefore, we have the following question.\\

\noindent  {\bf Question 2} Does there exist a locally finite, coarsely transitive graph of polynomial growth and finite asymptotic dimension that is not quasi-isometric to any finitely generated group?\\

In light of Trofimov's result, a positive answer to Question 2 would provide a positive answer to Question 1. If such a graph exists, it would be the most symmetric one could hope to have for a locally finite graph of polynomial growth which fails to be quasi-isometric to a finitely generated group. 
\bibliography{bib.bib}
\bibliographystyle{alpha}

\end{document}